\documentclass[12pt]{amsart}

\usepackage{amsmath}
\usepackage{amssymb, amstext, amsthm, amsfonts,euscript,amscd,mathrsfs}

\usepackage{graphicx}
\usepackage{cases}
\usepackage{enumitem}
\usepackage[final]{pdfpages}
\usepackage[toc,page]{appendix}
\usepackage{latexsym}
\usepackage[disable]{todonotes} 

\usepackage[normalem]{ulem}
\usepackage{array}
\usepackage{booktabs}

\DeclareMathOperator{\Real}{\mathbb{R}}

\newcommand{\N}{\mathbb{N}}

\newcommand{\jap}[1]{\!\left<#1\right>}
  \newcommand{\jxi}{\jap{\xi}}

\newcommand{\dx}{\partial_x}
\newcommand{\dxp}[1]{\partial_{x_{#1}}}

\newcommand{\dy}{\partial_y}
\newcommand{\dyp}[1]{\partial_{y_{#1}}}

\providecommand{\norm}[1]{\lVert#1\rVert}

\newcommand{\pd}{\partial}
\newcommand{\laplace}{\triangle}

\DeclareMathOperator{\dom}{dom}

\newcommand{\dt}{\pd_t}
 
\newcommand{\tld}[1]{\tilde{#1}}

\newcommand{\R}{\Real}

\newcommand{\eps}{\varepsilon}

\newcommand{\les}{\lesssim}

\newcommand{\dz}{\partial_{z}}

  \newcommand{\vo}{\vec{0}}

\newcommand{\D}{\mathscr{D}}

\newcommand{\Hormander}{{H\"{o}rmander}}

  \DeclareMathOperator{\supp}{supp}
   
\newcounter{theorem_counter}
\newtheorem*{thm*}{Theorem}
\newtheorem{thm}{Theorem}
\newtheorem{lem}[thm]{Lemma}
\newtheorem{prop}[thm]{Proposition}
\newtheorem{rem}[thm]{Remark}
\newtheorem{coro}[thm]{Corollary}

\newtheorem{defn}[thm]{Definition}

\usepackage{tabularx} 
\usepackage{booktabs} 

\usepackage[colorlinks=false]{hyperref}

\title{On the Necessity of Logarithmic Estimates for Hypoellipticity}
\author{Timur Akhunov}\address{Department of Mathematics and Computer Science\\
Wabash College\\
301 W Wabash Ave, Crawfordsville, IN 47933, USA}
\author{Lyudmila Korobenko}
\address{Mathematics Department\\
Reed College\\
3203 Southeast Woodstock Boulevard\\
Portland, OR 97202-8199, USA}

\begin{document}
\begin{abstract}
  This paper is focused on necessary conditions for hypoellipticity of an operator $L$ of the form $L=L_1(x)+g(x)L_2(y)$, where the operator $L_1$ is either elliptic or parabolic, $L_2$ is degenerately elliptic and $g(x)$ may itself vanish adding further degeneracy. First, we establish a logarithmic criterion: if the operator $L$ above is hypoelliptic and $L_1$ has a family of spectral solutions defined below, then the remaining part $L_2$ must gain a power of a logarithm of a derivative. Such a property can be thought of as a restriction on degeneracy of the operator $L_2$. We then use this criterion to examine degenerate elliptic and parabolic operators closing gaps between sufficiency and necessity that have been open since 1980s in three and higher dimensions.
\end{abstract}

%


\maketitle



\section{Introduction}
In this paper we investigate conditions in arbitrary dimensions under which local hypoellipticity of $L$, or smoothness of solutions of $Lu=f$ whenever $f$ is locally smooth (see Definition \ref{def-hypo} below). For a degenerate elliptic or parabolic operator $L$ the \Hormander\, bracket condition is a famous sufficient condition for hypoellipticity \cite{Hor67Brac,Ole-Rad73}. However, a sharp necessary and sufficient condition for hypoellipticity with smooth coefficients and infinite rates of degeneracy remains open.\footnote{ \cite{Bony69,MorXu-07,Amano-control} showed hypoellipticity implies controllability (or completeness of Carnot-Caratheodory metric) under fairly generic conditions. Controllability implies maximum principle \cite{MorXu-07,Amano-max}. Controllability is equivalent to \Hormander\, bracket condition for real analytic coefficients of $L$. However, with $C^\infty$ coefficients controllability is not sufficient for hypoellipticity \cite{Kusuoka-Strook85,Mor87}. Our work can be framed as quantifying the minimal rate of controllability.} While a vast literature emerged to address sufficiency criteria for hypoellipticity beyond the scope of {\Hormander} theory starting with a pioneering work of \cite{Fedii71}, fewer techniques have been developed to address optimality of such criteria particularly beyond two and three dimensions. For operators of the form $L(x,y)=L_1(x)+g(x)L_2(y)$ we develop a necessity criterion linked to a spectral condition for the operator $L_1$. The spectral estimates we work with have a certain resemblance with quantitative unique continuation \cite[Theorem 2.1]{Ign-Kuk12}, nodal sets of Laplacian \cite{Logunov2018}, quantitative  controllability and eigenfunction tunneling \cite{Lau-Lea2020}. As we discuss in Appendix \ref{sec:opt} applications of our criterion are optimal in a variety of settings.

Crucial to our work is Morimoto's superlogarithmic estimate
 \begin{defn}[Superlogarithmic estimate]\label{def-superlog}
   We say that an operator $L$ satisfies a superlogarithmic estimate of order $p>0$ near $y_0\in \R^m$ if there exists a compact set $K\Subset \R^m$ containing $y_0$, so that for every $\eps'>0$, there exists a constant $C_{\eps',K}$, such that the following estimate holds.
\begin{equation}\label{superlog}
||(\log\jap{\xi})^p\hat u(\xi)||^2 \leq \varepsilon'\,\mathrm{Re}(L u,u) + C_{\varepsilon',K} ||u||^2,\quad u\in C_{0}^{\infty}(K),
\end{equation}
where
     \begin{align}\label{jap-bracket}
   \jap{\xi}:=\sqrt{e^2+|\xi|^2}
 \end{align}
 \end{defn}

Morimoto showed in \cite{Mor87} that this estimate for $p=1$ is sufficient for the hypoellipticity of a symmetric second order operator $L$ in any dimensions by adapting probabilistic techniques of \cite{Kusuoka-Strook85} into PDE setting.\\

While superlogarithmic estimates are sufficient for hypoellipticity of many operators\footnote{
 In a related context, \cite{kohn02} established \eqref{superlog} and proved hypoellipticity of operators on CR manifolds. \cite{Mor-Xu03}, \cite{MorXu-07} investigated nonlinear solvability and hypoellipticity for operators satisfying \eqref{superlog}.}, they are not always necessary. To motivate cases when such an estimate is necessary and those when it is not, consider the simplest and extensively studied Fedii operator \cite{Fedii71}
 \begin{align}\label{Fedii2D}
 \begin{split}
  & L_{F}:=-\dy^2 - a(y)\partial_z^2, 
 \end{split}
\end{align}
and Kusuoka-Strook \cite{Kusuoka-Strook85} operator $L_{KS}=-\dx^2+L_F$. As we elaborate in more detail in section \ref{ref:Fed} based on  \cite{Mor-Xu03,Morimoto-Morioka97,Morimoto-Morioka97-Fedii}, in the special case of $a(y)=e^{-|y|^{-\alpha}}$, $L_F$ and $L_{KS}$ satisfy \eqref{superlog} if and only if $\alpha<\frac{1}{p}$ and violate {\Hormander}'s bracket condition for $\alpha>0$. Equivalently, the weight $a(y)=e^{-\frac{1}{|y|}}$ in \eqref{Fedii2D} leads to \eqref{superlog} for $p<1$ for $L_F$ and $L_{KS}$, but not $p=1$. Despite similar coefficients, hypoellipticity for $L_F$ and $L_{KS}$ for very degenerate $a(y)$ is quite different, as this table illustrates:
\renewcommand{\arraystretch}{1.4}
\begin{table}[!ht]
\centering
\caption{Simplest known examples that our work extends}
\label{table}
\renewcommand{\arraystretch}{1.4}
\begin{tabularx}{0.95\textwidth}{| l | X | X |}
\hline
\textbf{Operator} & \textbf{Minimal $p>0$ in \eqref{superlog} necessary for hypoellipticity } & \textbf{References that include and extend such examples} \\ \hline
$L_{F}$ & No restriction & \cite{Fedii71,Mor78,Morimoto-Morioka97-Fedii,Hosh-88,Kohn98,Christ01,AkhKorRios}\\
\hline
$-\dx^2+L_F$  & $p=1$ & \cite{Kusuoka-Strook85,Mor87,Hos87,Bell-Mohammed,Christ01,AkhKor24} \\ \hline
$\dt + L_F$ & $p=\tfrac{1}{2}$ & \cite{Hoshiro-89,Morimoto-Morioka97-Fedii}\\
\hline
$\dt - \dx^2+ L_F$& $p=1$ & \cite{Kusuoka-Strook85,Mor87,Hos87}\\
\hline
\end{tabularx}
\end{table}

All operators in the table as well as most of the known necessity results can be reduced to the form $L=\dt +L_2(x',D_{x'})$ or $\dx^2 + L_2(x',D_{x'})$, where $x$, $t$ and $x'$ are unrelated variables. We call it a clean split into $1+m$ dimensions. Most other necessity results are focused on specific examples in $3$ dimensions. For example \cite{Morimoto-Morioka97,Christ01} and references therein examine sharp sufficiency conditions for operators of the form $\dx^2+a(x)\dy^2 + b(x)\dz^2$. There has not been a systematic study of necessary conditions for sums of infinitely degenerate operators of the form $L_1(x,D_{x})+g(x)L_2(y,D_y)$ in arbitrary dimensions.\\

More specifically, we consider a degenerate parabolic/elliptic operator 
\begin{align}\label{Lgeneral}
  L(x,y)=L_1(x)+g(x)L_2(y)
\end{align}
where $g(x)\ge 0$
\begin{align}\label{L1}
  L_1u=-\sum_{j,k=1}^n a_{jk}(x)\dxp{j}\dxp{k}u +\sum_{j=1}^n a_j(x) \dxp{j}u + a_0(x)u
\end{align}
with $x\in \R^n$, $a_{jk}$, $a_j$ smooth real valued coefficients with $(a_{jk})$ a non-negative self-adjoint matrix, i.e.
\begin{align}\label{elliptic}
  \sum_{j,k=1}^n a_{j,k}(x) \eta_j \eta_k \ge a_e(x)|\eta|^2\ge 0
\end{align}
Note that condition \eqref{elliptic} allows for vanishing second derivatives and includes parabolic operators as well as pointwise degeneracy in some or all of the directions.
 Operator $L_2$ is defined similarly, but we assume, in addition, that $L_2$ is a symmetric operator:
 \begin{align}
  \label{L2}
  L_2 u = -\sum_{j,k=1}^m \dyp{j}[b_{jk}(y)\dyp{k}u] + b_0(y)u
\end{align}
where $y\in \R^m$,  $b_{jk}$ and $b_0$ are smooth real functions and $b_{jk}$ is a non-negative matrix, i.e. \eqref{elliptic} holds
\begin{align}\label{elliptic2}
  \sum_{j,k=1}^m b_{j,k}(y) \eta_j \eta_k = b_e(y)|\eta|^2\ge 0
\end{align}
From \eqref{L2} we have
 \begin{align}\label{symmetry}
   (L_2 u,v) = \sum_{j,k=1}^m (b_{jk}(y)\dyp{k}u,\dyp{j}v) + (b_0(y)u,v)=(u, L_2 v) \text{, for }  u,v\in C^\infty_0(R^m_y)
 \end{align}
  which gives $L_2$ real valued spectrum. 

 In \cite{AkhKor24}, the authors showed that \eqref{superlog} with $p=1$ is necessary for hypoellipticity of $L$  in the case when $L_1$ from \eqref{L1} is an elliptic one-dimensional operator with a more general structure than \cite{Hos87}, \cite{Mor87} and \cite{Christ01}. 
 As mentioned above, proceeding from $1+m$ to $n+m$ has not been done in general. The main innovation of this paper is a new spectral criterion that in addition to allowing results in higher dimensions allows a uniform treatment for degenerate elliptic and parabolic problems.\\

 Let us sketch the idea of this criterion. Let us consider $Lu=0$ on a small region around the origin. We perform a spectral projection for $L_2$ (see Appendix \ref{appendix-spectral} for a brief summary) and interpet $L_2$ as spectral parameter $\lambda$. With this change $Lu=0$ from \eqref{Lgeneral} becomes a generalized eigenvalue problem on a small box $Q_r$ around the origin:
 \begin{align}
   L_1w(x,\lambda)+g(x)\lambda w(x,\lambda)=0 \label{L-spec}
 \end{align}
 Suppose that we are able to construct a family of uniformly bounded spectral profiles $w(x,\lambda)$ for all  high frequencies $\lambda\gg 1$. Suppose in addition, that there exists a parameter  $p>0$ such that 
\begin{align}\label{lower}
  \log|w(0,\lambda)| \approx -(\sqrt{\lambda})^\frac{1}{p}
\end{align}
or equivalently $|\log|w(0,\lambda)||^p\approx \sqrt{\lambda}$.
{\bf Our criterion states that in such a scenario hypoellipticity of $Lu=0$ must imply \eqref{superlog} for $L_2$ with the same $p$.} Equivalently, if \eqref{superlog} fails for $L_2$, then either \eqref{lower} must fail for $L_1$ or hypoellipticity of $L$ is false. The rate $p$ of superlogarithmic gain depends on the strongest rate of non-degeneracy in complementary $x$ directions, which is connected to the spectral estimate \eqref{lower}.  We can think of the superlogarithmic estimate a minimal cost of smoothness in $y$ directions given spectral properties in complementary directions.
\subsection{Organization of the paper.} The remainder of the paper is organized as follows. Section \ref{sec:main} formally defines the spectral profile and states the main logarithmic criterion -- Theorem \ref{thm:main-general}. Section \ref{sec:appl} details the applications of our criterion, establishing the necessary logarithmic estimates for degenerate operators $L_1+g(x)L_2$, where $L_1$ is uniformly elliptic or parabolic operator.  Sections \ref{sec:unif} and \ref{sec:p+} are devoted to constructing the required spectral profiles for the elliptic and parabolic cases, respectively. Finally, Sections \ref{sec:gen:low} and \ref{sec:high} contain the proof of the main criterion, divided into low- and high-frequency analysis. Appendix \ref{appendix-spectral} gives a brief exposition of spectral projections needed for the paper. In Appendix \ref{sec:opt} we discuss operators that satisfy superlogarithmic estimates and demonstrate elliptic and parabolic cases, where our results are optimal. Appendix \ref{app:Sob} proves a variant of a Sobolev embedding needed for a version of the Main criterion with lower regularity.





\section{Main criterion}\label{sec:main}
For completeness, we start with a definition of local smooth hypoellipticity.
       \begin{defn}\label{def-hypo}
  The operator $L(x,y)$ is {\bf locally $C^\infty$ hypoelliptic} near $(x_0,y_0)$ if for every neighborhood $\Omega$ of $(x_0,y_0)$ and a function $f(x,y)$ smooth in $\Omega$, there exists a neighborhood $\Omega'\Subset\Omega$, such that any distribution $u$ satisfying $Lu=f$ in the distributional sense in $\Omega$,  must be smooth in $\Omega'$, i.e. $u\in C^\infty(\Omega')$.\\
  For brevity, we often simply refer to $L$ as hypoelliptic near $(x_0,y_0)$.
\end{defn}
Note, in particular, the region of smoothness of a distributional solution may depend on the operator $L$, region $\Omega$ and the specific function $f$, but not on the particular distributional solution itself.\\

  By translation, we may assume without loss of generality, that $(x_0,y_0)=(0,0)$, which will be our focus.\\ 

As mentioned in the introduction we perform a spectral projection for $L_2$ (see Appendix \ref{appendix-spectral}) and interpret $L_2$ as a spectral parameter $\lambda$. With this change for each $\lambda$, the equation $L_1u+g(x)L_2u=0$ becomes a spectral problem \eqref{L-spec}. Moreover, we may assume without loss of generality that  $\lambda\ge 1$ by Lemma \ref{lem:lambda-one} in Appendix \ref{appendix-spectral}.\\
Let $r>0$ and consider a box
\begin{align}\label{Qr}
  Q_r:=(-r,r)^n
\end{align}

Let $p>0$. The introduction motivated reduction of hypoellipticity to a spectral problem \eqref{Lspec} with uniformly bounded solutions that also have precise lower bounds for our main criterion. We slightly modify this set up here for technical convenience. More precisely, for our criterion we look for a point normalized solution $v(0,\lambda)=1$ with the following upper bound:
\begin{align}\label{spectral-log}
  \log^p|v(x,\lambda)| \les (\sqrt{\lambda}) \text{, on }\overline{Q}_r
\end{align}
We formalize this property with constant dependence in the following definition.
\begin{defn}\label{BVSP}
We say that the function $v(x,\lambda)$ is a {\bf spectral profile of order $p>0$} of operator $L=L_1+g(x)L_2$ on box $Q_{r_0}=(-r_0,r_0)^n$ for some $r_0>0$ if
\begin{enumerate}
\item There exist structural constants $c_0 > 0$ (depending only on the coefficient norms of $L_1$) and $C_r > 0$ (depending on $L_1$ and $r \in (0, r_0]$) such that the following holds:
     \begin{align}
    |v(x,\lambda)| \le C_r e^{c_0r\lambda^{\frac{1}{2p}} } \text{ on }\bar Q_r\label{criterion:UB}.
 \end{align}
  \item $v(x,\lambda)$ is a continuous distributional solution of
  \begin{align}
     L_1 v(x,\lambda) +  g(x)\lambda v(x,\lambda) =0  ,\text{ for } x\in Q_r\label{Lspec}
  \end{align}
  \item $v$ is point normalized, i.e.
  \begin{align}
      v(0,\lambda)= 1  \label{criterion:LWB}
  \end{align}
\end{enumerate}
  \end{defn}
 By {\bf solution} above we mean a function $v\in C^0(Q_r\times [1,\infty))$, which satisfies \eqref{Lspec} distributionally for all $\lambda\ge 1$, i.e. for every $\phi\in C^\infty_0(Q_r)$
 \begin{align}\label{Lsp-d}
   \int v(x,\lambda) (L_1^*+g(x)\lambda)\phi(x) dx =0
 \end{align}
  We are ready to state our criterion.
\begin{thm}\label{thm:main-general}[Logarithmic criterion]
Suppose that $L=L_1+g(x)L_2$ is locally $C^\infty$ hypoelliptic near $(0,0)$ and has a spectral profile of order $p>0$. Then $L_2$ must satisfy the logarithmic estimate for the same $p>0$ near $y=0$. That is for any $\eps'>0$ and any compact set $K$ contained in small enough open neighborhood of $y=0$, there is a constant $C_{\varepsilon',K}>0$, such that
\begin{equation*}
||(\log\jap{\xi})^p\hat u(\xi)||^2 \leq \varepsilon'\,\mathrm{Re}(L_2 u,u) + C_{\varepsilon',K} ||u||^2, \text{ for all }u\in C_{0}^{\infty}(K).
\end{equation*}
 	\end{thm}
 We formulated this logarithmic criterion to be general enough to apply to both degenerate elliptic and parabolic problems. This criterion is novel beyond $1+m$ dimensions, where \eqref{L-spec} is an ODE. In fact, conditions \eqref{criterion:UB}, \eqref{Lspec} and \eqref{criterion:LWB} are intimately tied to hypoellipticity. Point normalization \eqref{criterion:LWB} comes up in the proof of Theorem \ref{thm:main-general} due to the nesting of regions $\Omega'\Subset \Omega$ in the hypoellipticity Definition \ref{def-hypo}. \eqref{Lspec} is a spectral version of our main operator \eqref{Lgeneral}. Finally, the highest power $p$ in the spectral profile of $L_1$ in \eqref{criterion:UB} is the same as the highest necessary power of superlogarithmic estimate \eqref{superlog} for $L_2$.\\

  Note, that no assumption is made in the Theorem \ref{thm:main-general} on the weight $g(x)$ explicitly beyond the hypoellipticity of the operator $L=L_1+g(x)L_2$. In particular, $g(x)$ may vanish on as large or small of a set in some or all of the variables as long as the hypoellipticity of $L$ allows it. The only restriction we make is that the weight $g(x)$ is independent of $L_2$ variables $y$. We discuss some of these limitations in Appendix \ref{sec:opt}, particularly in \ref{ref:parab}.\\


 We prove Theorem \ref{thm:main-general} in sections \ref{sec:gen:low} and \ref{sec:high}, but we first examine some of the applications of this criterion.
\section{Applications of the Criterion}\label{sec:appl}
 Our first application is the case when the operator $L_1$ is uniformly elliptic.
\begin{thm}\label{main:thm}
  Suppose that $L_1$ in \eqref{Lgeneral} is uniformly elliptic. Suppose further that $L=L_1+g(x)L_2$ is hypoelliptic near $(x_0,y_0)$. Then $L_2$ must satisfy superlogarithmic estimate \eqref{superlog} of order $p=1$ near $y_0$, i.e. for all $\eps'>0$ and compact set $K$, there exists a constant $C_{\eps',K}$, such that
  \begin{equation*}
||\log\jap{\xi}\hat u(\xi)||^2 \leq \varepsilon'\,\mathrm{Re}(L_2 u,u) + C_{\varepsilon',K} ||u||^2,\quad u\in C_{0}^{\infty}(K).
\end{equation*}
\end{thm}
In section \ref{sec:unif} we construct a spectral profile of order $p=1$ in the sense of Definition \ref{BVSP} and reduce the proof to Theorem \ref{thm:main-general}. \\

To consider parabolic results we relabel the $n$-th dimension $t$ and assume $L_1$ has the following form
\begin{align*}
  L_1 =a_n(x,t)\dt-\sum_{j,k=1}^{n-1} a_{jk}(x,t)\dxp{j}\dxp{k} +\sum_{j=1}^{n-1} a_j(x,t) \dxp{j} + a_0(x,t)
\end{align*}
We focus on uniformly parabolic examples here, i.e. where in the remaining $n-1$ dimensions $a_{jk}$ matrix is uniformly elliptic and $a_n(0,0)\neq 0$. Replacing $t\to -t$ if needed, we may assume, without loss of generality, that $a_n(0,0)> 0$.  Provided that we work in a small neighborhood of $(0,0)$ we can ensure that $a_n(x,t)>0$. Dividing \eqref{Lgeneral} by $a_n(x,t)$ we may assume, without loss of generality that $a_n(x,t)\equiv 1$.\\


For $L_1$ that is uniformly parabolic in higher dimensions, the necessary logarithmic condition is identical to the uniformly elliptic case.
\begin{thm}\label{parab:n}
Let
\begin{align}\label{parab-eq}
  L_1 =\dt-\sum_{j,k=1}^{n-1} a_{jk}(x,t)\dxp{j}\dxp{k} +\sum_{j=1}^{n-1} a_j(x,t) \dxp{j} + a_0(x,t)
\end{align}
Suppose that $L=L_1+g(x)L_2$ is hypoelliptic near $(x_0,t_0,y_0)=\vo$ and $a_{jk}$ is uniformly elliptic, i.e. $a_{jk}(\vo)\ge c>0$ in the sense of positive definite matrices. Then $L_2$ must satisfy superlogarithmic estimate of order $p=1$, i.e. for all $\eps'>0$ and compact set $K$, there exists a constant $C_{\eps',K}$, such that
  \begin{equation*}
||\log\jap{\xi}\hat u(\xi)||^2 \leq \varepsilon'\,\mathrm{Re}(L_2 u,u) + C_{\varepsilon',K} ||u||^2,\quad u\in C_{0}^{\infty}(K).
\end{equation*}
\end{thm}
Meanwhile, in $1+m$ dimensional case or $n=1$, we show weaker superlogarithmic estimate for $p=\frac{1}{2}$ is necessary for hypoellipticity 
\begin{thm}\label{parab1}
   Suppose that $L=\dt+a_0(t)+g(t)L_2$ is hypoelliptic near $(t_0,y_0)$. Then $L_2$ must satisfy superlogarithmic estimate \eqref{superlog} of order $p=\frac{1}{2}$
  \begin{equation*}
||(\log\jap{\xi})^{\frac{1}{2}}\hat u(\xi)||^2 \leq \varepsilon'\,\mathrm{Re}(L_2 u,u) + C_{\varepsilon',K} ||u||^2,\quad u\in C_{0}^{\infty}(K).
\end{equation*}
\end{thm}
We prove Theorems \ref{parab:n} and \ref{parab1} in \ref{sec:p+} by reducing them to the logarithmic criterion Theorem \ref{thm:main-general}. Meanwhile in section Appendix \ref{sec:opt}, we show that Theorems \ref{main:thm}--\ref{parab1} are optimal in a range of examples.

\section{Uniformly elliptic $L_1$}\label{sec:unif}
In this section we show that the extrapolation problem \eqref{Lspec}-(\ref{criterion:UB}) of order $p=1$ has a solution in the sense of Definition \ref{BVSP} in the case when the operator $L_1$ is uniformly elliptic.\\

\begin{prop}\label{prop:u-est}
Consider the following boundary value problem for $\lambda\ge 1$
\begin{align}\label{L1BV}
  \begin{cases}
    L_1 u(x,\lambda) +  g(x)\lambda u(x,\lambda) =0  ,\quad x\in Q_r:=(-r,r)^n\\
      u(x,\lambda) =e^{c \sqrt{\lambda}(x_1-r)} \text{ on } \partial Q_r
  \end{cases}.
\end{align}
  There exists $c_0=c_0(\inf a_{11},\norm{a_1}_{L^\infty},\norm{a_0^-}_{L^\infty})$ and $r_0=r_0(\inf a_{11},\norm{a_1}_{L^\infty},\norm{a_0^-}_{L^\infty})>0$ such that for $0<r\le r_0$ there exists a unique $C^{2,\alpha}$ solution of \eqref{L1BV} which satisfies
\begin{equation}\label{bounds}
e^{c \sqrt{\lambda}(x_1-r)} \le u(x,\lambda)\le 2
\end{equation}
for $c\ge c_0$.
\end{prop}

Let
\begin{equation}\label{def:tilde-l}
  \tld L:= L_1+ g(x)\lambda
\end{equation}
To prove this Proposition we need the following Generalized Maximum principle from \cite{Prot84} that does not assume a sign of $a_0$ in \eqref{L1} beyond boundedness. Note that our use of a minus sign in front of the elliptic term in \eqref{L1} is the opposite of the convention used in \cite{Prot84} and we adjust the signs in the statement below.
\begin{thm}\cite[Chapter 2, Theorem 10]{Prot84}\label{GMP}
  Suppose that $v(x)$ satisfies the differential inequality
  \begin{align}\label{Lsup}
    \tld L v\le 0
  \end{align}
  in the domain $Q_r$, where $\tld L$ is defined by (\ref{def:tilde-l}) and $L_1$ is uniformly elliptic in the sense of \eqref{L1}. If there exists a function $w$, such that
  \begin{align*}
    w>0 \text{ on }\bar{Q_r}\\
    \tld Lw\ge 0 \text{ in } Q_r
  \end{align*}
  then $\frac{v}{w}$ cannot attain a nonnegative maximum in $Q_r$ unless it is constant.
\end{thm}
\begin{proof}[Proof of Proposition \ref{prop:u-est}]
First, by \cite[Chapter 3 Theorem 1.2]{Lady} there exists $r_0>0$, more precisely $r_0=r_0(\inf a_{11},\norm{a_1}_{L^\infty},\norm{a_0^-}_{L^\infty})$ such that for $r\le r_0$ uniqueness holds for $C^{2,\alpha}$ solutions of \eqref{L1BV} for each $\lambda\ge 1$. Next, \cite[Chapter 3 Theorem 1.1]{Lady} implies existence of a $C^{2,\alpha}$ solution to \eqref{L1BV} on $Q_r$.\\
We now prove that the solution satisfies (\ref{bounds}). First we construct a barrier function $w$ which satisfies conditions of Theorem \ref{GMP}. Define
\begin{align}\label{w}
  w=3 -\exp(\beta (r-x_1))
\end{align} where $\beta>0$ is chosen below.
Let $r\leq 1$, then for every $\lambda>0$ we have
\begin{align}\label{prelim-est}
(L_1 +  \lambda g(x))w&=3(a_0+g(x)\lambda)+\left(a_{11}\beta^2+a_1\beta-a_0-g(x)\lambda\right)\exp(\beta (r-x_1))\notag \\
&\ge -3\norm{a_0^-}_{L^\infty(Q_1)} + \left((\inf_x a_{11}(x)\beta^2-\norm{a_1^-}_{L^\infty(Q_1)})\beta -\norm{a_0^+}_{L^\infty}\right)\exp(\beta (r-x_1))\notag\\
&\quad+g(x)\lambda(3 -\exp(\beta (r-x_1)))\notag\\
&\ge ((\inf_x a_{11}(x)\beta-\norm{a_1}_{L^\infty(Q_1)})\beta -4\norm{a_0}_{L^\infty(Q_1)})+g(x)\lambda(3 -\exp(\beta (r-x_1))),
\end{align}
where we used $\exp(\beta (r-x_1))\ge 1$ on $Q_r$ since $\beta>0$.
We now show $(\inf_x a_{11}(x)\beta-\norm{a_1}_{L^\infty})\beta -4\norm{a_0}_{L^\infty(Q_1)}\geq 0$ for an appropriately chosen $\beta>0$.  Define
\begin{align*}
	\beta_0=\frac{\norm{a_1}_{L^\infty}+1}{\inf_x a_{11}(x)} \text{ and } \beta_1=4\norm{a_0}_{L^\infty},
\end{align*}
and let $\beta=\max\{\beta_0,\beta_1\}>0$. For $x\in \overline{Q_r}$ we have $x_1\in [-r,r]$, and thus
\begin{align*}
	3 -\exp(2\beta r)\le w(x)\le 3-1 \text{ in } \overline{Q_r}.
\end{align*}
Moreover, $(\inf_x a_{11}(x)\beta-\norm{a_1}_{L^\infty})\beta -4\norm{a_0}_{L^\infty(Q_1)}\geq 0$ in $\overline{Q_r}$ by the choice of $\beta$.
Let
\begin{align*}
	r_0=\min\left\{\frac{\log 2}{2\beta},1\right\},
\end{align*}
then for $0<r\le r_0$ we have
\begin{equation}\label{w-bound}
	1\le w(x)\le 2 \text{ in } \overline{Q_r}.
	\end{equation}
Combining with (\ref{prelim-est}) and the estimate $(\inf_x a_{11}(x)\beta-\norm{a_1}_{L^\infty})\beta -4\norm{a_0}_{L^\infty(Q_1)}\geq 0$ obtain above, we conclude
\begin{align*}
	(L_1 +  \lambda g(x))w\ge 0\text{ in } Q_r.
\end{align*}
Applying Theorem \ref{GMP} to the solution of \eqref{L1BV}, $u(x)$, we obtain
\begin{equation}\label{u-bound}
\sup_{Q_r} \frac{u^+(x)}{w(x)} \le \sup_{\partial Q_r}\frac{u^+(x)}{w(x)}.
\end{equation}

Next, define
  \begin{align}\label{u-low}
  u_{L}=e^{c \sqrt{\lambda}(x_1-r)}
  \end{align}
  and note that we have pointwise estimate
  \begin{align*}
  e^{-2c \sqrt{\lambda}r}\le u_{L}(x)\le 1 \text{ when } |x_1|\le r.
  \end{align*}
  Moreover, substituting $u_L$ into $\tld L$ we obtain
  \begin{align*}
  \tld Lu_L= [-a_{11} c^2 \lambda + a_1 c \sqrt{\lambda} +\lambda g+ a_0]u_L.
  \end{align*}
  As $\lambda\ge 1$ in hypothesis of Definition \ref{BVSP}, choosing $$c\ge c_0=c_0(\inf a_{11},\norm{a_1}_{L^\infty},\norm{a_0}_{L^\infty},\norm{g}_{L^\infty})$$ large enough we can assure that
  \begin{align}\label{Lu-low}
  \tld L u_L \le 0.
  \end{align}
  The boundary condition in \eqref{L1BV} implies that $u(x)=u_L(x)$ on $\partial Q_r$. Meanwhile, by construction $u_L\le 1$ on $\partial Q_r$. Combining these estimates with (\ref{w-bound}) and (\ref{u-bound}) we obtain.
  \begin{align*}
    \sup_{Q_r} \frac{u^+(x)}{w(x)}  \le 1.
  \end{align*}
Since we showed that $w\leq 2$, we obtain the estimate $\sup_{Q_r} u^+(x)\leq 2$. This gives the upper bound claimed in Proposition \ref{prop:u-est}.\\
We now apply Theorem \ref{GMP} to $v(x)=u_L(x)-u(x)$. We compute
  \begin{align*}
    \tld L v = \tld L u_L\le 0,
  \end{align*}
  and therefore,
  \begin{align*}
    \sup_{Q_r} \frac{(u_L(x)-u(x))^+}{w(x)} \le \sup_{\partial Q_r}\frac{(u_L(x)-u(x))^+}{w(x)}=0.
  \end{align*}
Using the boundary condition in \eqref{L1BV} we obtain
\begin{align*}
    u_L(x)-u(x)\le 0.
  \end{align*}
  This gives $u(x)\geq u_L(x)=e^{c \sqrt{\lambda}(x_1-r)}$ and concludes the proof.
\end{proof}
We conclude the section by defining
  \begin{align*}
    v(x,\lambda)=\frac{u(x,\lambda)}{u(0,\lambda)}
  \end{align*}
  By linearity of \eqref{L1BV}, $v$ solves \eqref{Lspec}. \eqref{criterion:LWB} is true by construction. Finally, \eqref{criterion:UB} is immediate from the Proposition \ref{prop:u-est}.\\
Finally, because $v$ is a classical solution of \eqref{Lspec}, it satisfies \eqref{Lsp-d}.
\section{Parabolic case}\label{sec:p+}
\subsection{One dimensional case}
To prove Theorem \ref{parab1}, by Theorem \ref{thm:main-general}, the argument reduces to the following.
\begin{prop}\label{1d-parab}
The family
  \begin{align*}
  v(t,\lambda)=\frac{w(t,\lambda)}{w(0,\lambda)}
\end{align*}
where $w$ solves the following initial value problem\footnote{Since we are addressing regularity around $(0,0)$, our initial values are posed before $t=0$}
\begin{align*}
\begin{cases}
  \dt w + [a_0(t)  + g(t)\lambda] w = 0 \text{ on } (-T,T)\\
  w(-T)=1
\end{cases}
  \end{align*}
satisfies Definition \ref{BVSP} with $p=\frac{1}{2}$
\end{prop}
\begin{proof}
  The solution of the system above is
  \begin{align*}
    w(t,\lambda)=\exp(-\int_{-T}^t [a_0(t') + g(t') \lambda ] dt'),
  \end{align*}
  and thus $w$ is bounded above:
  \begin{align*}
    w(t,\lambda)\le \exp((t+T)\norm{a_0^-}_{L^\infty})
  \end{align*}
  where $a_0^-$ is the negative part of $a_0(t)$. Similarly, $w$ is bounded below
  \begin{align*}
    w(t,\lambda)\ge \exp\left(-(t+T)[\norm{g}_{L^\infty}\lambda+\norm{a_0^+}_{L^\infty}]\right).
  \end{align*}
  Therefore
  \begin{align*}
    0<v(t,\lambda)\le C(T,a_0,g)e^{c\lambda T}\quad \text{for all}\  t\in (-T,T),
  \end{align*}
where $c=c(\norm{g}_{L^\infty},\norm{a_0}_{L^\infty})$.
  As we relabeled $r$ from Definition \ref{BVSP} as $T$, this concludes the proof.
  \end{proof}
\subsection{Higher dimensional case}
In this section we show that the extrapolation problem \eqref{Lspec}-(\ref{criterion:UB}) of order $p=1$ has a solution in the sense of Definition \ref{BVSP} in the case when $L_1$ is a parabolic operator with a uniformly elliptic part. \\

We define as in the previous section (replacing $n$ by $n-1$)
\begin{align}\label{Lt}
  \tld L= -\sum_{j,k=1}^{n-1} a_{jk}(x,t)\dxp{j}\dxp{k} +\sum_{j=1}^{n-1} a_j(x,t) \dxp{j} + a_0(x,t) + g(x,t)\lambda
\end{align}
Let $Q_r=(-r,r)^n$ as before, and for $T>0$ define a parabolic cylinder and its boundary as in \cite[Chapter 7.1]{Ev98}
\begin{align*}
  Q_{r,T}=Q_r\times (-T,T] \text{ and } \Gamma_T=\overline{Q_{r,T}}-Q_{r,T},
\end{align*}
where just like in Proposition \ref{1d-parab} the initial moment in time is $t=-T$.
As in elliptic case considered in section \ref{sec:unif} we use barriers to establish existence and appropriate upper and lower bounds. Define
\begin{align}\label{Lt0}
\tld L_0= -\sum_{j,k=1}^{n-1} a_{jk}(x,t)\dxp{j}\dxp{k} +\sum_{j=1}^{n-1} a_j(x,t) \dxp{j}.
\end{align}
Then $\tld L = \tld L_0 + f(x,t,u)$, where $f(x,t,u)=[a_0(x,t) + g(x,t)\lambda]\cdot u$ is linear and hence locally Lipschitz in $u$.

Consider the following initial boundary value problem for $\lambda\ge 1$
\begin{align}\label{L1IBV}
  \begin{cases}
    \dt u + \tld L_0 u = - (a_0+\lambda g)u  \text{ on } Q_{r,T}\\
      u(x,t)=u_L \text{ on }\Gamma_T
  \end{cases}
\end{align}

We will show that its classical solution, rescaled appropriately, satisfies Definition \ref{BVSP} for $L_1=\dt + \tld L$ with $p=1$.

\begin{prop}\label{parab:close}
	There exists a unique classical solution, $u(x,t,\lambda)$, of \eqref{L1IBV} for every $\lambda\ge 1$. Moreover, $v(x,t,\lambda)$ defined by
	\begin{align}\label{v-def}
	v(x,t,\lambda):=\frac{u(x,t,\lambda)}{u(0,0,\lambda)},
	\end{align}
	satisfies Definition \ref{BVSP} for $L_1=\dt + \tld L$ with $p=1$.
\end{prop}

\begin{proof}
	
	We will use the approach of \cite{Pao} to show there exists a solution to (\ref{L1IBV}) with appropriate upper and lower bounds. In order to do this, we first construct upper and lower solutions, see Definition 3.1 in \cite[Chapter 2]{Pao}. Let $\tilde{u}:=u_L$ where $u_L(x)$ defined in \eqref{u-low} with $c\ge c_0=c_0(\inf a_{11},\norm{a_1}_{L^\infty},\norm{a_0}_{L^\infty})$ satisfies \eqref{Lu-low}, and thus
	\begin{align*}\label{BVt}
	  \begin{cases}
		\dt \tilde{u} + \tld L_0 \tilde{u} \leq - (a_0+\lambda g)\tilde{u}  \text{ on } Q_{r,T}\\
		\tilde{u}(x,t)=u_L \text{ on }\Gamma_T
	\end{cases}
	\end{align*}
	This shows that $\tilde{u}$ is a lower solution to (\ref{L1IBV} for any fixed $r>0$\footnote{Unlike the elliptic case there is no restriction on $r$ here}..
	For the upper solution, let $\alpha=||a_0^-||_{L^\infty}\geq 0$, and define the space independent function $\hat{u}:=e^{\alpha (t+T)}$. We have using the fact that $g\geq 0$
	\[
	\dt \hat{u} + \tld L\hat{u}\ge 0,
	\]
	so $\hat u$ is an upper-solution of \eqref{Lspec}.\\	
Since $f(x,t,u)=[a_0(x,t) + g(x,t)\lambda]\cdot u$	is locally Lipschitz, we can apply Theorem 4.1 \cite[Chapter 2]{Pao} to conclude that  there exists a unique classical solution $u$ of (\ref{L1IBV}) satisfying $\tilde{u}\leq u\leq \hat{u}$ in $\overline{Q_{r,T}}$. Recalling the definitions, we thus have writing $u(x,t,\lambda)=u(x,t)$
\begin{align}\label{parab:bounds}
e^{c\sqrt{\lambda}(x_1-r)}\leq u(x,t,\lambda)\leq e^{\alpha(t+T)}.
\end{align}
	Now consider $v(0,0,\lambda)$ defined in (\ref{v-def}). By construction $v(0,0,\lambda)=1$.\\
	From the lower bound in \eqref{parab:bounds} for $(x,t)\in Q_{r,T}$
	\begin{align*}
	|v(x,t,\lambda)|\le  e^{2\sqrt{\lambda}r}|u(x,t)|\le e^{2\sqrt{\lambda}r+2\alpha T}
	\end{align*}
	Finally, by linearity $v$ solves \eqref{Lspec}.
	\end{proof}

By Theorem \ref{thm:main-general}, $L_2$ must satisfy superlogarithmic estimate \eqref{superlog} of order $1$. This concludes the proof of Theorem \ref{parab:n}.

\section{Proof of Theorem \ref{thm:main-general}: low frequencies}\label{sec:gen:low}
To prove Theorem \ref{thm:main-general}, we examine how hypoellipticity of the operator $L$ in \eqref{Lgeneral} interacts with the spectral projections $E_\lambda$ for $L_2$ (defined in Appendix \ref{appendix-spectral}). We state our strategy informally before elaborating on details. Our focus here is on the $y$ variables, where $(i\xi)^\alpha \hat u(\xi)= \widehat{\dy^\alpha u}(\xi)$. Meanwhile on the domain of $L_2$, $L_2$ acts as $\lambda$. We distinguish between two frequency regions.
\begin{itemize}
  \item Low frequency. As our argument below explains for $\lambda\les \log^{2p} \jap{\xi}$ we will use hypoellipticity of the full operator $L$. Intuitively, for this frequency threshold condition we have
\begin{align*}
  e^{\eps \lambda^{\frac{1}{2p}}}\les \jxi^\eps
\end{align*}
\item High frequency $\lambda \gtrsim \log^{2p}\jap{\xi}$. In this region we will show in section \ref{sec:gen:finish} that $\log^p\jap{\xi}$ can be controlled by the operator.
\end{itemize}
We begin the low frequency argument as follows. Let $B$ is the Friedreich extension of the operator $L_2$ to its maximal domain in $L^2$, see Appendix \ref{appendix-spectral}.

 \begin{prop}[Full operator solutions]\label{prop:PDE:full}
 Suppose the spectral problem \eqref{Lspec} satisfies Definition \ref{BVSP} for $r_0>0$.
Let $\eps>0$ and define $r=\min\{\frac{\eps}{8c_0},r_0\}$ with $c_0$ from \eqref{criterion:UB}.
Let $v(x,\lambda)$ be a solution of \eqref{Lspec} on $Q_r$ satisfying Definition \ref{BVSP}. Let $u(y)\in L^2_y(K)$ be such that
\begin{align}\label{spectral:assumption}
  e^{\eps B^{\frac{1}{2p}}} u\in L^2
\end{align}  Define $w$ by
 \begin{align}\label{spectral-construction-Laplace}
   w(x,y) = v(x,B) u(y) = \int_1^\infty v(x,\lambda) dE_{\lambda} u(y).
 \end{align}
  Then
 \begin{align}\label{w:est}
   \norm{w}_{L^2_{Q_r\times K}}\le C(\eps,r) \norm{e^{\eps B^{\frac{1}{2p}}} u}_{L^2(K)}
 \end{align}
  Finally,   $w$ is a distributional solution of
\begin{align}\label{PDE-solution}
  \begin{cases}
    Lw=0 \text{  on } Q_r\times K;\\
    w(0,y)=u(y)
  \end{cases}
\end{align}

 \end{prop}
 \begin{proof}
 	We first show (\ref{w:est}).
  Using \eqref{spectral-construction-Laplace} and \eqref{criterion:UB} we have for $\varepsilon$ sufficiently small
  \begin{multline}\label{pde:full:est}
    \norm{w(x,\cdot)}_{L^2_{K}}\le \int_1^\infty|v(x,\lambda)|^2 d\norm{E_{\lambda} u}_{L^2_y(K)}\\
    \le \int_1^\infty  (C(r)e^{c_0r\lambda^{\frac{1}{2p}} })^2  d\norm{E_{\lambda} u(y)}^2 \\
   \le C(r)\int_1^\infty e^{\frac{\eps}{4}\lambda^{\frac{1}{2p}}}  d\norm{E_{\lambda} u(y)}^2\le C(r)\norm{e^{\eps B^{\frac{1}{2p}}} u}_{L^2(K)}<\infty.
  \end{multline}
  Therefore, $u\in D(v(x,B))$ and \eqref{w:est} holds.\\

  By \eqref{criterion:LWB}, $v(0,B)$ is an identity operator on $L^2$, so $w(0,y)=u(y)$ for $u\in L^2_y(K)$.\\

  It remains to show that $w$ is a distributional solution of \eqref{PDE-solution}. Let $\phi\in C^\infty_0(Q_r\times K)$. Note that as $w\in L^2_{loc}$ the $(w,L^*\phi)$ pairing makes sense as an integral. By operational calculus
  \begin{align*}
    (w,L^*\phi)=\int_{Q_r} \int_1^\infty v(x,\lambda)   (dE_{\lambda}u(y),L_1^* \phi)  dx \\ +  \int_{Q_r} \int_1^\infty g(x,t)\lambda v(x,\lambda)   (dE_{\lambda}u(y), \phi)  dx
  \end{align*}
  where the pairings are well-defined as $\lambda \le C(\eps,r)e^{\frac{\eps}{4}\lambda^{\frac{1}{2p}}}$ and we allowed room for a slightly bigger exponent in the last line of \eqref{pde:full:est}.\\
  By Fubini we interchange the order of integration and apply \eqref{Lsp-d}. Hence \eqref{PDE-solution} holds distributionally.
 \end{proof}
 \subsection{Uniform estimates on solutions using hypoellipticity}
 In this section, we establish quantitative Sobolev space estimates for solutions $w$ of $Lw=0$ for $L$ that is locally hypoelliptic. Arguments follow
the outline of Propositions 15 and 16 in \cite{AkhKor24} adapted to work with Definition \ref{BVSP} rather than a spectral ODE in one dimension.
 \begin{lem}\label{lem:X}  Let $\Omega$ be a neighborhood of the origin. Define
     \begin{align}\label{hypo:space}
       X:=\{ w \in L^2(\Omega)\mid L w = 0 \text{ in the distributional sense on }\Omega
        \}
     \end{align}
      Then $X$  is a closed subspace of $L^2(\Omega)$.
     \end{lem}
     \begin{proof}
        Let $w_n \to w \in L^2(\Omega)$ with $L w_n=0$ distributionally. Then
  \begin{align*}
    (Lw,\phi) = (w, L^*\phi) = \lim_n (w_n, L^*\phi)=0 \text{ for }\phi \in C^\infty_0(\Omega)
  \end{align*}
  Thus $Lw=0$ distributionally and $w\in X$.
     \end{proof}

     \begin{lem}\label{closed-graph}[Closed Graph] Let $s> 0$.
     Let $\Omega$ be a neighborhood of the origin.   Suppose the operator $L$ from \eqref{Lgeneral} is locally hypoelliptic at $(0,0)$ and let $\Omega'\Subset\Omega$ be as in the Definition \ref{def-hypo}.
 Define the operator $T$ by\begin{align*}
      \hspace{-20pt}  T:X \to H^s(\Omega') ,\quad Tw = w
      \end{align*}
     Then $T$ has a closed graph. In particular, it is bounded, so there exists a constant $\tld C=\tld C(\Omega,\Omega',r,r')$
     \begin{align}
       \label{closed-graph:est} \norm{w}_{H^s(\Omega')}\le \tld C \norm{w}_{L^2(\Omega)} \text{ for } w \in X
     \end{align}
  \end{lem}
  \begin{proof}
    By Lemma \ref{lem:X} $X$ is Banach space. Hypoellipticity guarantees that $T$ is well-defined. \\

    To show that the operator $T$ has a closed graph, consider a sequence $u_n \to u \in X$ and $Tu_n\to v \in H^s(\Omega')$. Passing to a subsequence, $u_n(x,y) \to u(x,y)$ and $u_n(x,y) \to v(x,y)$ for almost every $(x,y)\in \Omega'$. Hence $u=v$ almost everywhere in $\Omega'$, and thus $u\in H^s(\Omega')$. Moreover, since $u\in X$, we have $Tu=u=v$. This shows that the graph of $T$ is closed.
 By the closed graph theorem,   c.f. \cite{YosFunct} p.80, $T$ is continuous, and hence satisfies \eqref{closed-graph:est}.
  \end{proof}
\begin{prop}\label{prop:Sob}
 Suppose $s>\frac{m+n}{2}+1$ and $\{0\}\times K_y\Subset     \Omega$. Then for $w\in H^s(\Omega)$ there holds
  \begin{align}\label{Sob-high}
    \norm{w(0,\cdot)}_{H^1(K_y)}\le \tld C(s_1,s_2,K_y,\Omega)\norm{w}_{H^s(\Omega)}
  \end{align}
\end{prop}
\begin{proof}
By Sobolev embedding
  \begin{align*}
   \norm{w(x,\cdot)}_{C^1(K_y)}\le \tld C(s_1,s_2,K_y,\Omega)\norm{w}_{H^s(\Omega)}
  \end{align*}
  setting $x=0$ and integrating in $y$ implies \eqref{Sob-high}.
\end{proof}
Proposition \ref{prop:Sob} can modified to $s>\frac{n}{2}$, if one wishes to lower the regularity of the arguments needed to prove Theorem \ref{thm:main-general}.
\begin{prop}\label{prop:Sob:part}[Sobolev Embedding nD]
	Let $n\ge 1$, $s>\frac{n}{2}$ and let $\Omega=Q_r\times Q\subset \R^n_{x}\times \R^m_y$ be a neighborhood of the origin. Suppose that $w(x,y)\in H^{s}(\Omega)$. Let $\Omega'\Subset \Omega$ and
	consider a neighborhood $\Omega'':=Q_{r''}\times Q''$ small enough, so that $\Omega''\Subset \Omega'$.\\
Let $s_1$ be a number satisfying $\frac{n}{2}<s_1<s$ and define $s_2:=s-s_1> 0$, so that $s=s_1+s_2$. Then $x\mapsto w(x,\cdot)$ is a continuous uniformly bounded function from $Q_r\to H^{s_2}_y(Q'')$ with
	\begin{align}
		\label{hypo:est} \norm{w(0,\cdot)}_{H^{s_2}(\Omega'')}\le \tld C(s_1,s_2,\Omega,\Omega',\Omega'')\norm{w}_{H^s(\Omega)}
	\end{align}
\end{prop}
We defer the proof to Appendix \ref{app:Sob}.
     \subsection{Projections and solutions}
       Recall that to invoke Proposition \ref{prop:PDE:full} and construct a solution of \eqref{PDE-solution}, \eqref{spectral:assumption} has to be satisfied. This assumption is too strong for $L^2$ functions and even $C^\infty_0$. However, the results above are relevant for $\lambda\les \log^{2p} \jap{\xi}$ as we discussed at the start of Section \ref{sec:gen:low}.
       \begin{lem}\label{lem:proj}
         Let the operator $P_j$ be the spectral projection for the operator $L_2$ defined in \eqref{P-j-def} in Appendix \ref{appendix-spectral}. Then \eqref{spectral:assumption} holds for $P_ju$, whenever $u\in L^2_y$. In fact,
         \begin{align}\label{Pj-exp}
           \norm{ e^{\eps B^{\frac{1}{2p}}} P_ju(y) }_{L^2_y}^2\le 6 e^{\eps {e^{{\frac{j+1}{2p}}}}}\norm{P_ju}^2
         \end{align}
       \end{lem}
       \begin{proof}
         Let $u\in L^2$ and define $u_j=P_ju$. By Proposition \ref{lem:LP-operator} in Appendix \ref{appendix-spectral}.
 \begin{align*}
    \norm{ e^{\eps B^{\frac{1}{2p}}} u_j(y) }_{L^2_y}^2\le 2\sum_{|j'-j|\le 1} e^{\eps e^{ {\frac{j'}{2p}}}} \norm{P_{j'} P_j u}^2
 \end{align*}
     Summing over all $j'\in\mathbb{N}$ concludes the proof.
     \end{proof}
      \begin{prop}\label{lem:necess:Hs}
Let $s>\frac{m+n}{2}+1$ and $j\in\N$. Suppose the spectral problem \eqref{Lspec} satisfies Definition \ref{BVSP} for $r_0>0$ and ${\frac{1}{2p}}>0$. Let $\eps>0$ and set $r$ as in Proposition \ref{prop:PDE:full}. Let $P_j$ be the spectral projection for the operator $L_2$ defined in \eqref{P-j-def}. Let $Q\subset \R^m_y$ be a cube containing the origin.\\
Suppose $L$ from  \eqref{Lgeneral} is {hypoelliptic near $0$}.  Then there exists a cube $Q'\Subset Q$ and a constant $\tld C=\tld C(\eps,s_1,s_2,p,Q)$, both independent of $j$, so that
       \begin{align}\label{Hs-tldL}
       \norm{P_j u}_{H^{1}(Q')}  \le \tld C  \exp(\eps e^{\frac{j}{2p}}) \norm{P_j u}_{L^2(Q)}
     \end{align}
     for all $u(y)\in L^2_y(Q)$ and all $j\ge 0$.
\end{prop}
\begin{proof}
  By Lemma \ref{lem:proj}, \eqref{spectral:assumption} is valid for $u_j=P_ju$ and hence hypothesis of Proposition \ref{prop:PDE:full} is satisfied for $u_j$. Proposition \ref{prop:PDE:full} implies that there exists $w_j$ that solves \eqref{PDE-solution} with
\begin{align}\label{projection:at0}
  w_j(0,y)=u_j(y)=P_ju(y)
\end{align}
Let $Q''$ be a cube sandwiched inside $Q'$ and $Q$, i.e. $Q'\Subset Q''\Subset Q$. We invoke Proposition \ref{prop:Sob}. By \eqref{Sob-high}
 \begin{align*}
   \norm{w_j(0,\cdot)}_{H^{1}(Q')} \le \tld C \norm{w_j}_{H^s(Q_{\frac{r}{2}}\times Q'')}
 \end{align*}
where the constant $\tld C$ is independent of $j$ and $u$. By Lemma \ref{closed-graph} we further estimate
\begin{align*}
  \norm{w_j}_{H^s(Q_{\frac{r}{2}}\times Q'')} \le \tld C\norm{w_j}_{L^2(Q_r\times \Omega)}
\end{align*}
Combining the last two estimates we deduce
\begin{align}
  \label{closed:Hs}
   \norm{P_ju}_{H^{1}(Q')} \le \tld C \norm{w_j}_{L^2(Q_r\times \Omega)}
\end{align}
It remains to estimate $w_j$ by its boundary data using Proposition \ref{prop:PDE:full}. More specifically, we apply \eqref{w:est} for $w=w_j$:
\begin{align*}
  \norm{w_j}_{L^2(Q_r\times Q)} \le C\norm{e^{\eps B^{\frac{1}{2p}}} u_j}_{L^2(Q)}
\end{align*}
Combining this estimate with \eqref{Pj-exp} and \eqref{closed:Hs} concludes the proof.
\end{proof}
\begin{rem}\label{rem:Sobolev}
  We can lower the regularity $H^s$ in Proposition \ref{lem:necess:Hs}, which comes from the Sobolev embedding in Proposition \ref{prop:Sob} to $s>\frac{n}{2}$ from Proposition \ref{prop:Sob:part}, with $s_2=s-s_1>0$ for $s>s_1>\frac{n}{2}$.
       \begin{align}\label{Hs-tldL-lower}
       \norm{P_j u}_{H^{s_2}(Q')}  \le \tld C  \exp(\eps e^{\frac{j}{2p}}) \norm{P_j u}_{L^2(Q)}
     \end{align}
   The only change in the proof of Proposition \ref{prop:Sob:part} is to replace the Sobolev estimate \eqref{Sob-high} with \eqref{hypo:est}.
\end{rem}

\section{Proof of Theorem \ref{thm:main-general}: high frequencies}\label{sec:high}

\begin{lem}[Interpolation] \label{lem:interp}
    Let $\eps>0$, $s_2>0$ and $p>0$. For each $\xi\in \R^n$ let $R=R(\xi)$ be defined by
    \begin{align}\label{R-xi}
      e^R = \left(\frac{s_2\log\jxi}{2\eps}\right)^{2p}.
    \end{align}
   Consider a sequence of functions $\alpha_j(y) \in H^{s_2}_y$ for $j=0, 1, 2 \ldots$ Then
    \begin{align}
      \label{low-j}& \left\Vert\sum_{j\le R(\xi)} \log^p\jxi \widehat{\alpha_j}(\xi)\right\Vert_{L^2_\xi}^2 \le C(\eps^{-1},p,s_2)\sum_{j=0}^\infty e^{-2\eps  e^{\frac{j}{2p}}}\norm{\alpha_j}_{H^{s_2}_y}^2\\
      \label{high-j}& \left\Vert\sum_{j\ge R(\xi)} \log^p\jxi \widehat{\alpha_j}(\xi)\right\Vert_{L^2_\xi}^2 \le C(s_2,p)\eps^{2p}\sum_{j=0}^\infty e^{j}\norm{\alpha_j}_{L^2_y}^2
    \end{align}
  \end{lem}


  \begin{proof}
  The argument is a modification of Lemma 21 from \cite{AkhKor24} that allows for $p\neq 1$.
   Note that \eqref{R-xi} is equivalent to
   \begin{align}\label{xi-R}
     \jxi^{s_2} = \exp(2\eps e^{\frac{R}{2p}}) \text{ and } R=2p(\log\log\jxi + \log s_2 - \log(2\eps))
   \end{align}
   By the Cauchy-Schwartz inequality
    \begin{align}\label{Cauchy-low}
      \sum_{j\le R} \log^p\jxi| \widehat{\alpha_j}(\xi)| \le \left(\sum_{j\le R} \frac{\log^{2p}\jxi}{\jxi^{2s_2}} e^{2\eps e^{\frac{j}{2p}}}\right)^{\frac{1}{2}}\cdot \left( \sum_{j\le R}\frac{\jxi^{2s_2}}{e^{2\eps e^{\frac{j}{2p}}}} |\widehat{\alpha_j}|^2\right)^{\frac{1}{2}}
    \end{align}
    The terms of the first sum are increasing in $j$, so the sum can be estimated
    \begin{align*}
      \sum_{0\le j\le R} \frac{\log^{2p}\jxi}{\jxi^{2s_2}} e^{2\eps e^{\frac{j}{2p}}} \le \frac{\log^{2p}\jxi}{\jxi^{2s_2}} R \cdot e^{2\eps e^{\frac{R}{2p}}}
    \end{align*}
    Hence by \eqref{xi-R}
    \begin{align*}
       \sum_{j\le R} \frac{\log^{2p}\jxi}{\jxi^{2s_2}} e^{2\eps e^{\frac{j}{2p}}} \le \frac{\log^{2p}\jxi}{\jxi^{s_2}} 2p(\log\log\jxi + \log s_2 - \log(2\eps)).
    \end{align*}
    Taking a supremum over all $\xi$ gives
    \begin{align*}
       \sum_{j\le R} \frac{\log^{2p}\jxi}{\jxi^{2s_2}} e^{2\eps e^{\frac{j}{2p}}} \le  C(\eps^{-1},p,s_2)
     \end{align*}
     Returning with this estimate to \eqref{Cauchy-low} and integrating in $\xi$ we get
    \begin{align*}
      \norm{\sum_{j\le R} \log^{p}\jxi \widehat{\alpha_j}(\xi)}_{L^2_\xi}^2 \le  C(\eps^{-1},p,s_2) \int \sum_{j\le R}\frac{\jxi^{2s_2}}{e^{2\eps e^{\frac{j}{2p}}}} |\widehat{\alpha_j}|^2 d\xi
    \end{align*}
  Observe
     \begin{align*}
      \sum_{j\le R}\frac{\jxi^{2s_2}}{e^{2\eps e^{\frac{j}{2p}}}}|\widehat{\alpha_j}|^2 \le  \sum_{j=0}^\infty\frac{\jxi^{2s_2}}{e^{2\eps e^{\frac{j}{2p}}}}|\widehat{\alpha_j}|^2.
     \end{align*}
Interchanging the sum and integration for $\int \sum_{j=0}^\infty\frac{\jxi^{2s_2}}{e^{2\eps e^{\frac{j}{2p}}}}|\widehat{\alpha_j}|^2 d\xi$ using Fubini concludes the proof of \eqref{low-j}.

  Similarly for \eqref{high-j}, by Cauchy-Schwartz
  \begin{align*}
    \left(\sum_{j\ge R} \log^{p}\jxi| \widehat{\alpha_j}(\xi)| \right)^2\le \left( \sum_{j\ge R} \frac{\log^{2p}\jxi}{e^j}\right) \cdot \left( \sum_{j\ge R} e^j |\widehat{\alpha_j}|^2\right)
  \end{align*}
  The first sum is a geometric series in $j$, so
  \begin{align*}
    \sum_{j\ge R} \frac{\log^{2p}\jxi}{e^j} = \frac{\log^{2p}\jxi}{e^R}\frac{1}{1-1/e}
  \end{align*}
  Using \eqref{R-xi} gives
  \begin{align*}
    \frac{\log^{2p}\jxi}{e^R}\frac{1}{1-1/e}\le \left(\frac{2\eps}{s_2}\right)^{2p}\cdot 3
  \end{align*}
We conclude that
 \begin{align}\label{interpolate:high-pt}
    \left(\sum_{j\ge R} \log^{p}\jxi| \widehat{\alpha_j}(\xi)| \right)^2\le C(p,s_2)\eps^{2p} \left( \sum_{j\ge R} e^j |\widehat{\alpha_j}|^2\right)
  \end{align}
 For $C(p,s_2)= (\frac{2}{s_2})^{2p}\cdot 3$. Now extend the sum $$\sum_{j\ge R} e^j|\widehat{\alpha_j}|^2\le \sum_{j\ge 0} e^j |\widehat{\alpha_j}|^2$$ and integrate \eqref{interpolate:high-pt} in $\xi$ to obtain
   \begin{align*}
    \norm{\sum_{j\ge R} \log^{p}\jxi \widehat{\alpha_j}}_{L^2}^2 \le C(p,s_2)\eps^{2p}\int \sum_{j=0}^\infty e^j|\hat \alpha_j|^2 d\xi.
  \end{align*}
 Applying Fubini concludes \eqref{high-j}.
  \end{proof}

  \subsection{Proof of Theorem \ref{thm:main-general}}\label{sec:gen:finish}
 \begin{proof}[Proof of Theorem \ref{thm:main-general}]
 Suppose the spectral problem \eqref{Lspec} satisfies Definition \ref{BVSP} for $r_0>0$ and $p>0$.
 Let $\eps'>0$ be given.
 Let $s_2=1$ and $s_1> \frac{m+n}{2}$. Let $C(s_2,p)$ be the constant from \eqref{high-j}. Define $\eps$ by
 \begin{align}\label{eps-eps}
 \eps^{2p}=\frac{\eps'}{C(s_2,p)e}
 \end{align}
 Set $r$ as in the Proposition \ref{prop:PDE:full}.

   Let $u(y)\in C^\infty_0(\Omega')$, and define $\alpha_j =  P_ju$. Since $P_j$'s are a resolution of identity by Proposition \ref{lem:LP-operator}, we can decompose
   \begin{align*}
   u=  \sum_{j\ge 0} P_j u= \sum_{j\ge 0} \alpha_j
   \end{align*}
  Let $R$ be defined by (\ref{R-xi}). By the triangle inequality $$ \norm{\log^{p}\jxi\hat u}_{L^2} \le  \norm{\log^{p}\jxi \sum_{j\le R}\hat \alpha_j}_{L^2}+  \norm{\log^{p}\jxi \sum_{j> R} \hat \alpha_j}_{L^2}$$
  By Cauchy-Schwartz
  \begin{align*}
    \norm{\log^{p}\jxi \hat u}_{L^2}^2\le 2\left(\norm{\log^{p}\jxi \sum_{j\le R} \hat\alpha_j }_{L^2}^2 + \norm{\log^{p}\jxi \sum_{j> R}\hat \alpha_j }_{L^2}^2\right)
  \end{align*}
  We are now set to use Lemma \ref{lem:interp} for our choice of $\alpha_j$:
  \begin{multline}\label{log-necess-split}
    \norm{\log^{p}\jxi \hat u}_{L^2}^2  \le C(s_2,p)\eps^{2p}\sum_{j=0}^\infty e^j\norm{ P_ju }_{L_2}^2 \\
    + C_\eps\sum_{j=0}^\infty e^{-2\eps e^{\frac{j}{2p}}}\norm{ P_j u}_{H^{s_2}}^2 := I+ II
  \end{multline}

  We estimate $I$ and $II$ separately. For $I$ we estimate the first sum by \eqref{LP-oper} in Proposition \ref{lem:LP-operator} from Appendix \ref{appendix-spectral} for $f(\lambda)=\sqrt{\lambda}$. Since $f(\lambda)>0$ and increasing for $\lambda\ge 1$
\begin{align*}
 I= C(s_2,p)\eps^{2p}\sum_{j=0}^\infty e\cdot |\sqrt{e^{j-1}}|^2\norm{P_ju }_{L_2}^2\\
  \le C(s_2,p)e\eps^{2p}\norm{\sqrt{B} u}^2
\end{align*}
By \eqref{eps-eps} we observe
\begin{align*}
  I\le \eps' \norm{\sqrt{B} u}^2.
\end{align*}
Using $(Bu,u) = (L_2 u, u)$ and the fact that $B$ is an extension of the operator $L_2$ and $u\in C^\infty_0$ we obtain
\begin{align*}
  I\le \eps' (L_2u,u)
\end{align*}

For $II$ we observe from Proposition \ref{lem:necess:Hs}
\begin{align*}
  \norm{P_j u}_{H^{s_2}}^2 \le C(\eps,s,p) e^{2\eps e^{\frac{j}{2p}}}\norm{P_j u}_{L^2}^2
\end{align*}
Hence by (\ref{log-necess-split}) and Proposition \ref{lem:LP-operator}
\begin{align*}
  II \le C(\eps,s,p) \sum_j \norm{ P_j u}_{L^2}^2 \le C_\eps \norm{u}^2
\end{align*}
Combining $I$ and $II$ into \eqref{log-necess-split} we obtain
\begin{align*}
  \norm{\log^{p}\jxi u}_{L^2(Q')}^2 \le \eps' (L_2 u, u) + C_{\eps'} \norm{u}^2_{L^2(\Omega)}\hspace{20pt}\qedhere
\end{align*}
 \end{proof}

\appendix
 \section{Spectral projections}\label{appendix-spectral}
      We state relevant spectral analysis in this appendix. We follow \cite{AkhKor24} closely.\\

      By \eqref{L2}, \eqref{elliptic2}, $L_2$ is an operator bounded below. That is given $\Omega \Subset \R^n$ there is a constant $c_\Omega\in \R$ so that
 \begin{align}\label{global-lower-bound}
 	(L_2 u, u) \ge c_{\Omega}\norm{u}^2 \text{ for all } u\in C_0^\infty(\Omega)
 \end{align}
 \begin{lem}\label{lem:lambda-one}
Without loss of generality we may assume that $L_2$ is a positive operator:
 	\begin{align}\label{L2:positive}
 		(L_2u,u) \ge \norm{u}^2
 	\end{align}
 \end{lem}
\begin{proof}
   	By replacing $L_2$ with $L_2-c_{\Omega}+1$ and $L_1$ with $L_1+g(x)(c_{\Omega}-1)$, \eqref{Lgeneral} remains the same and \eqref{global-lower-bound} holds for $c\ge 1$.
\end{proof}
Let $Q$ be an open cube containing the origin in $\R^n_y$. We extend the domain of $L_2$ in \eqref{L2} from $C^\infty_0(Q)$ to a unique maximal domain inside $L^2(Q)$ to make it self adjoint. More precisely, let $B$ be the Friedreich extension of $L_2$ and $E_B(\lambda):L^2(Q)\to L^2(Q)$ be the spectral projection associated to $B$ (c.f. Theorem 2 on p. 317 of \cite{YosFunct}).
                Spectral measure allows us define functions of the operator $f(B)$                 for any measurable function $f$ as follows:
\begin{align}\label{spectral-function}
f(B)u = \int_{-\infty}^\infty f(\lambda) dE_B(\lambda) u
\end{align}
with the domain $$\dom f(B) = \{ u\in L^2(Q): \norm{f(B)u}^2:= \int_{-\infty}^\infty |f(\lambda)|^2 d ||E_B(\lambda) u||^2 <\infty \}.$$

For the analysis below it is convenient to localize solutions to specific frequencies. Littlewood-Paley decomposition is ideal for that. Let $\Phi$ be a fixed cut-off function satisfying
\begin{align}\label{Phi}
  \Phi(\lambda)\in C^\infty_0([-e,e]), \,\, \Phi(\lambda)\equiv 1 \text{ for }|\lambda|\le 1 \text{ and } 0\le \Phi \le 1.
\end{align}
Moreover, we assume that $\Phi$ is even and nonincreasing in $[0,e]$. Now, for any integer $j$ let
 \begin{align*}
   \Phi_j(\lambda)\equiv \Phi(\lambda\cdot e^{-j}).
 \end{align*}
Then $\Phi_j(\lambda)=1$ for $|\lambda|\le e^j$ and $\Phi_j(\lambda)=0$ for $|\lambda|\geq e^{j+1}$. We further define cut-offs localized to $|\lambda|\approx e^j$ as follows:
      \begin{align}\label{psi-j}
      \psi_j(\lambda) = \Phi_j(\lambda)-\Phi_{j-1}(\lambda) \text{ for } j\ge 0,
      \end{align}
 so that $\supp{\psi_j(|\lambda|)}=[e^{j-1},e^{j+1}]$ and $0\le \psi_j\le 1$.\\
With $\psi_j$ as above, we define a frequency localization adapted to the operator via \eqref{spectral-function} \begin{align}\label{P-j-def}
        P_j u := \psi_j(B)u= \int \psi_j(\lambda) dE_B(\lambda)u 
      \end{align} for $j\in \N$. Key properties of the operator $P_j$ are summarized in the following result.
   \begin{prop}[Lemma 26 in \cite{AkhKor24}]  \label{lem:LP-operator}  Let $P_j$ be defined by \eqref{P-j-def}. Then for any $u\in L^2(\Omega)$ there holds
      \begin{align*}
        & u=\sum_{j=0}^\infty P_j u \hspace{15 pt} \norm{u}^2\approx \sum_{j=0}^\infty \norm{P_ju}^2;\\
        & \int P_j u P_{j'} v dx = 0 \text{ for } |j-j'|>1 \text{ and } v\in L^2
      \end{align*}
   Finally, for $f:[1,\infty)\to \mathbb{R}_+$ nondecreasing
\begin{align}\label{LP-oper}
  \sum_{j=0}^\infty |f\left(e^{j- 1}\right)|^2 \norm{P_j u}^2 \le \norm{f(B) u}^2 \le 2 \sum_{j=0}^\infty |f\left(e^{j+ 1}\right)|^2 \norm{P_j u}^2
\end{align}
and the norm is finite if and only if the right side of the equation is finite.
      \end{prop}
\section{Optimality and some extensions} \label{sec:opt}
Our appendix roughly follows the table \ref{table} on p. \pageref{table} - sufficiency of elliptic examples, then parabolic ones with $p=\frac{1}{2}$ and $p=1$.\\
When we combine \cite[Theorem 1]{AkhKorRios} and Theorem \ref{main:thm}, when both are applicable, we obtain the following:
       \begin{coro}\label{AKR}
       Suppose $L=L_1+g(x)L_2$, $L_1$ is uniformly elliptic\footnote{Sufficiency of hypoellipticity in \cite{AkhKorRios} applies more generally to a symmetric operator $L_1$ satisfying \eqref{superlog} for $p=1$ and generalizes \cite{Mor87,Christ01}}, $L_2$ is symmetric degenerate elliptic operator, i.e. satisfies \eqref{L2} and the weight $g(x)$
   \begin{align*}
     \text{$g(0)=0$ and $g(x)>0$ for }x\neq 0 \text{\textbf{ or }} g(0)>0
   \end{align*}
   Then $L$ is hypoelliptic near $(0,0)$ if and only if $L_2$ satisfies \eqref{superlog} for $p=1$.
\end{coro}
This implies that superlogarithmic estimate \eqref{superlog} for $p=1$ is optimal for hypoellipticity of operators of the form $L=L_1+g(x)L_2$ at least in the case of elliptic $L_1$ and empty or isolated vanishing set for $g(x)$. Because no restriction is made on the decay rate of $g(x)$ in Corollary \ref{AKR} the full operator $L$ may not satisfy \eqref{superlog} for any $p>0$. We can think of Corollary \ref{AKR} as generalizing the first two lines in table.

\subsection{Superlogarithmic estimates and stopping time condition}\label{ref:Fed}
As mentioned in the introduction, Fedi\u{\i} \cite{Fedii71} found sufficient conditions for hypoellipticity an operator $L_F=\dy^2 + a(y)\dx^2$ beyond the bracket condition. Here, we review the literature for pointwise and average conditions on the coefficient $a(y)$ that are equivalent to the superlogarithmic estimate \eqref{superlog} for $L_F$. We also discuss some of the generalizations to higher dimensions. This allows us to examine implications for the coefficients of concrete operators for which our criterion and its applications are relevant.\\

 Fedii operator $L_F$ was shown hypoelliptic (we focus at the origin), whenever the smooth function $a(y)$ satisfies the following stopping time condition
\begin{align}\label{A-1}
  a_I=\int_I a(y) dy>0, \text{ for all intervals }I\subset [-\delta,\delta]
\end{align}
c.f. \cite[Theorem 0]{Morimoto-Morioka97-Fedii} and  \cite[Theorem 6.1]{Morimoto-Morioka97}. This condition is a generalization of monotonicity and symmetry of $a(y)$ that is often assumed
\begin{align}\label{a-mono}
  a(0)=0, \,\, a(-y)=a(y) \text{ and }a'(y)>0, \text{ for } y>0
\end{align}
In particular, condition \eqref{A-1} allows for both infinite rates of degeneracy like $a(y)=e^{-|y|^{-\alpha}}$ for any $\alpha>0$, but also varnishing of $a(y)$ to be along nowhere dense, but positive Lebesgue measure sets.\footnote{E.g. vanishing on a fat-Cantor or Smith–Volterra–Cantor set} 

 Kusuoka and Strook examined the following the following three dimensional operator
 \begin{align}\label{L3}
   L_{KS}=-\dx^2+L_F =-\dx^2-\dy^2 - a(y)\partial_z^2
 \end{align}
  that adds $L_1=-\dx^2$ to Fedii operator $L_{F}$. Under monotonicity assumption \eqref{a-mono}
and smooth square root of $a(y)$ they showed that the condition
 \begin{align}\label{Fedii-rate}
   \lim_{y\to 0} |y|^{\frac{1}{p}}\log a(y) = 0
 \end{align}
is necessary and sufficient for hypoellipticity of $L_{KS}$. Their argument was based on probabilistic techniques and  condition \eqref{Fedii-rate}. In that language \eqref{Fedii-rate} can be interpreted as a quantitative stopping time rate for the degenerate brownian motion. 
Superlogarithmic estimate \eqref{superlog} was formulated soon after \cite{Kusuoka-Strook85} by \cite{Hos87,Mor87} to understand the hypoellipticity of this operator using PDE techniques. For the monotone case, \eqref{Fedii-rate} was shown equivalent to superlogarithmic estimate \eqref{superlog} for $L_F$ and $L_{KS}$ for $p=1$. In the case of non-monotone function $a(y)$ the most general stopping time condition that is necessary and sufficient for hypoellipticity of $L_{KS}$ is the $M_1$ condition of Morimoto-Morioka, which we state for a general index $p>0$.
 \begin{defn}
 We say that a function $a(y)$ satisfies condition $M_p$ provided it satisfies \eqref{A-1} and
    \begin{align}\label{Mp}
  \inf_{0<\delta\le \delta_0} \left( \sup \left\{ |I|^\frac{1}{p} |\log a_{3I}| \, : \,  a_{3I} < \delta \right\} \right) = 0
\end{align}
where $3I$ denotes the interval with the same center as $I$ but with length
three times that of $I$.
 \end{defn}
 One can think of this condition as the quantitative version of a vanishing rate for a non-monotone function $a(y)$ in \eqref{Fedii-rate}, just like the non-negative average condition $a_I>0$ is the average version of pointwise vanishing $a(0)=0$ and $a(y)>0$ for $y\neq 0$.  \cite{Morimoto-Morioka97} shows that $M_1$ condition is equivalent to \eqref{Fedii-rate} under the ``monotonicity'' condition \eqref{A-1}, but allows more oscillation in general. We state the \cite[Prop 5.1]{Mor-Xu03} here:


\begin{prop}[\cite{Mor-Xu03} ]\label{Fedii-log}
  Suppose $a_I>0$ for all intervals near $0$. Then $M_p$ condition is equivalent to superlogarithmic estimate \eqref{superlog} for $L_F$ for the same $p$.
\end{prop}
  Moreover, $M_1$ is necessary and sufficient for hypoellipticity of $L_{KS}$ \cite[Theorem 1]{Morimoto-Morioka97}. Meanwhile, a special case of \cite[Proposition 1]{Morimoto-Morioka97-Fedii} shows $M_{\frac{1}{2}}$ is necessary and sufficient to hypoellipticity of $L=\dt + L_F$. We use these properties to extend the latter result in \ref{ref:parab}.\\

An application of Theorem 
 \cite[Prop 5.2]{Mor-Xu03} extends the quantitative stopping time condition $M_p$ in \eqref{Mp} to arbitrary dimension for sums of squares operators $\laplace_X = \sum_{j=1}^m X_j^* X_j$ and shows sufficiency of such conditions for \eqref{superlog}. Commutators do play a role in restricting the vanishing rate of the coefficients to satisfy superlogarithmic estimate. For example, the operator $L_4 = \partial_{y_1}^2 + y_1^2 \partial_{y_2}^2 +a(y_2)\dz^2$ from \cite[Proposition 4]{Mor87} with a monotone $a(y_2)$ with a smooth square root, satisfies the superlogarithmic estimate with $p=1$ if and only if $a(y_2)$ satisfies $M_{2}$ condition or equivalently \eqref{Fedii-rate} for $p=2$. For $a(y)=e^{-|y|^{-\alpha}}$ it means that $\alpha<\frac{1}{2}$, as opposed to  $\alpha<1$ for Fedii and K-S operators $L_F$ and $L_{KS}$ respectively. In light of this discussion, superlogarithmic estimate \eqref{superlog} can be thought of as a quantitative rate of control on the degeneracy of the ellipticity of the operator $L$ beyond the bracket condition.\footnote{Or quantifying stopping time rate of controllability in \cite{Amano-control}}

\subsection{Parabolic results}\label{ref:parab}
As discussed in the introduction superlogarithmic estimate with $p=1$ is stronger than the one with $p=\frac{1}{2}$. We begin with the case of $p=\frac{1}{2}$, where Theorem \ref{parab1} settles the sharpness for a wide class of parabolic perturbations of Fedi\u{\i}'s operator $L_F$ in $1+2$ dimensions. More precisely, \cite[Proposition 3]{Morimoto-Morioka97-Fedii} gives the following sufficient condition for hypoellipticity:
\begin{prop}[\cite{Morimoto-Morioka97-Fedii}]
  Consider $L=\dt - g(t) (\dy^2 + a(y)\partial_z^2)$ with $g(t)$ as well as $a(t)$ satisfying \eqref{A-1}, i.e. $g_I>0$ and $a_I>0$ on every interval $I$. Suppose further that $a(y)$ satisfies $M_p$ condition \eqref{Mp} for $p=\frac{1}{2}$. Then $L$ is hypoelliptic.
\end{prop}
This result extends \cite{Hoshiro-89} which examined $g(t)=1$ with monotone $a(y)$ satisfying \eqref{a-mono}. Proposition \ref{Fedii-log} shows that $M_{\frac{1}{2}}$ condition is equivalent to the superlogarithmic estimate \eqref{superlog} for $p=\frac{1}{2}$ for $L_F$. In other words, $L_F$ satisfying \eqref{superlog} is sufficient for hypoellipticity of $L=\dt + g(t) L_F$ from \eqref{Fedii2D}. Unlike many results in \cite{Morimoto-Morioka97-Fedii} this result did not have a matching necessary condition in the paper. Theorem \ref{parab1} implies that superlogarithmic estimate for $p=\frac{1}{2}$ for $L_F$ is necessary for hypoellipticity of $L=\dt + g(t) L_F$. We arrive at:
\begin{coro}
  Consider the operator $L=\dt + g(t) L_F$ with $g(t)$ (and $a(y)$) satisfying \eqref{A-1}. The following are equivalent
  \begin{itemize}
    \item L is hypoelliptic
    \item $L_F$ satisfies \eqref{superlog} for $p=\frac{1}{2}$
    \item The $M_p$ condition \eqref{Mp} holds for $a(y)$ for $p=\frac{1}{2}$
  \end{itemize}
\end{coro}
We next consider operators of the form
\begin{align}\label{Hosh:eq}
 L_H:=\dt + L_1(x,D_x) + g(x)L_2(y,D_y)
\end{align}
for uniformly parabolic $\dt+L_1$, and where we denote $\tld L_2=L_1 + g(x)L_2$ so that $L_H=\dt + \tld L_2$. In this setting both Theorems \ref{parab1} and \ref{parab:n} apply and imply necessity of \eqref{superlog} - with $p=1$ for $L_2$ and $p=\frac{1}{2}$ for $\tld L_2$. For simplicity, we consider a non-degenerate $g(x)$, i.e. $g(0)\neq 0$. In that setting a special case of \cite[Theorem 1]{Hosh-88}\footnote{\cite[Theorem 1]{Hosh-88} allows a general $\tld L_2$ that satisfies \eqref{superlog} for $p=1$}
states that a superlogarithmic estimate \eqref{superlog} for $L_2$ with $p=1$ is sufficient for hypoellipticity of $L_H$. Note, that \cite[Theorem 1]{Hosh-88} hypothesis requires $C^1_t \D'$ solutions rather than distributions in space time to get smooth solutions, which is slightly different than the standard Definition \ref{def-hypo}. Up to this technical point, we see that for parabolic operators of the form \eqref{Hosh:eq} superlogarithmic estimate for $p=1$ is sharp for hypoellipticity:
\begin{coro}\label{Hosh:opt}
  Suppose the function $g(x)$ is non-degenerate and operator $L_1$ is uniformly elliptic. The the operator $L_H$ is hypoelliptic for $C^1_t \D'$ solutions if and only if $L_2$ satisfies \eqref{superlog} for $p=1$.
\end{coro}
We resolve the technical point here.
\begin{proof}[Proof of Corollary \ref{Hosh:opt}]
   Reduction of Theorem \ref{parab:n} to the Logarithmic Criterion Theorem \ref{thm:main-general} requires no change, since we constructed classical spectral solutions of \eqref{parab-eq} in $t,x$. Our proof of the Theorem \ref{thm:main-general} in sections \ref{sec:gen:low} and \ref{sec:high} does not rely on a specific derivative being a priori bounded or unbounded. Hence it applies for solution that is a priori $u \in C^1_tC^2_xL^2_y$ for spectral theory in $y$ to apply by \eqref{P-j-def}, whereas the proof of Theorem \ref{main:thm} considered $u\in C^0_{t,x}L^2_y$. \qedhere
 \end{proof}

\section{Lower regularity with Sobolev embedding}
\label{app:Sob}
A version of Proposition \ref{prop:Sob:part}, was proved in one dimension in \cite[Lemma 18]{AkhKor24}. We restate the earlier result that we need for the general case here.
\begin{prop}\label{Sobolev-1D}[Sobolev Embedding 1D]
	Let
	\begin{align}\label{s=s_1+}
		\tld s_1>\frac{1}{2}\text{ and }s_2\ge 0
	\end{align}
	Suppose $u(x',z)\in H^{s_1+s_2}(\R_{x'}\times \R^k_z)$. Then $x'\mapsto u(x',\cdot)$ is a continuous uniformly bounded function from $\R\to H^{s_2}_y$ with
	\begin{align*}
		\norm{u}_{L^\infty_{x'}H^{s_2}_z}\le C_{s_1,s_2}\norm{u}_{H^{s_1+s_2}_{x',z}}
	\end{align*}
	In particular,
	\begin{align}\label{Sobolev-estimate}
		\norm{u(0,\cdot)}_{H^{s_2}_z}\le C_{s_1,s_2}\norm{u}_{H^{s_1+s_2}_{x',z}}
	\end{align}
\end{prop}
We now prove the $n$ dimensional version.
\begin{proof}[Proof of Proposition \ref{prop:Sob:part}]
	First, consider $u(x,y)\in H^{s_1+s_2}(\R^{n}_{x}\times \R^m_{y})$.
	Since $s_1>n/2$ we can partition $s_1=\sum_{j=1}^n s_{1,j}$ with each $s_{1,j}>\frac{1}{2}$. We then iterate \eqref{Sobolev-estimate} one coordinate of $x=(x_1,\ldots x_n)\in \R^n$ at a time. More precisely, for the first step we apply Proposition \ref{Sobolev-1D} with $x'=x_1\in \R$ and $z=(x_2,\ldots x_n,y)\in \R^{n+m-1}$ to obtain
	\[
	\norm{u(0,\cdot)}_{H^{s_{2}+s_1-s_{1,1}}(\R^{n+m-1})}\le C_{s_1,s_2,s_{1,1}}\norm{u}_{H^{s_1+s_2}_{x,y}(\R^n_x\times \R^m_y)}.
	\]
	At step two with $u(0,x',z)\in H^{s_{2}+s_1-s_{1,1}}(\R_{x'}\times\R^{n+m-2}_z)$ , $x'=x_2$, and ${z=(x_3,\dots, x_n,y)\in \R^{n+m-2}}$ we get combining with step one
	\begin{align*}
		\norm{u(0,0,\cdot)}_{H^{s_{2}+s_1-s_{1,1}-s_{1,2}}(\R^{n+m-2})}&
		\leq C_{s_1,s_2,s_{1,1},s_{1,2}}\norm{u(0,\cdot)}_{H^{s_{2}+s_1-s_{1,1}}(\R_{x'}\R^{n+m-1})}\\
		&\leq C_{s_1,s_2,s_{1,1},s_{1,2}}\norm{u}_{H^{s_1+s_2}_{x,y}(\R^n_x\times \R^m_y)}
	\end{align*}
	Repeating $n$ times, we obtain
	\begin{align}\label{Sobolev-nD}
		\norm{u(0,\cdot)}_{H^{s_2}(\R^m_y)}\le C_{s_1,s_2}\norm{u}_{H^{s_1+s_2}_{x,y}(\R^n_x\times \R^m_y)}.
	\end{align}	
	Now consider $w(x,y)\in H^{s_1+s_2}(\Omega)$. Let $\psi_x\in C^\infty_0(\R^n_x)$ and $\psi_y\in C^\infty_0(\R^m_y)$ with $\psi_x(x)\equiv 1$ on $Q_{r''}$ and $\psi_y(y)\equiv 1$ on $Q''$. Further, suppose $0\le \psi\le 1$. Finally, we need $$\supp\psi_x\times\supp \psi_y \subset \Omega'.$$
	Define $\psi(x,y):=\psi_x(x)\psi_y(y)$. We apply \eqref{Sobolev-nD} to $u=\psi w$ to get
	\begin{align*}
		\norm{\psi(0,\cdot) w(0,\cdot)}_{H^{s_2}_y(\R^m_y)}\le C_{s_1,s_2}\norm{\psi w}_{H^s_{x,y}(\R^n_x\times \R^m_y)}
	\end{align*}
	Since $\psi(0,y)=\psi_y(y)$ with $\psi_y\equiv 1$ on $Q''$ we obtain
	\begin{align*}
		\norm{w(0,\cdot)}_{H^{s_2}_y(Q'')}\le \norm{\psi(0,\cdot)w(0,\cdot)}_{H^{s_2}_y(\R^m_y)}.
	\end{align*}
	Finally, by definition we have
	\begin{align*}
		\norm{\psi w}_{H^s_{x,y}(\R^n_x\times \R^m_y)} \le C(\Omega,\Omega'',s)\norm{w}_{H^s(\Omega')}.
	\end{align*}
	Combining the last three estimates gives
	\begin{align*}
		\norm{w(0,\cdot)}_{H^{s_2}_y(Q'')}\le C\norm{w}_{H^s(\Omega')}.
	\end{align*}
\end{proof}

\bibliographystyle{amsalpha}
\bibliography{Hypo-paper-2025-temp}
\end{document}